\documentclass{amsart}
\usepackage{graphicx}

\newtheorem{theorem}{Theorem}
\newtheorem{lemma}{Lemma}
\newtheorem{corollary}{Corollary}
\theoremstyle{remark}

\numberwithin{equation}{section}

\begin{document}

\title{An Asymptotic Formula for the Chebyshev Theta Function}

\author{Aditya Ghosh}
\address{Indian Statistical Institute, Kolkata, India}
\email{ghoshadi26@gmail.com}

\date{}

\begin{abstract}
 Let $\{p_n\}_{n\ge 1}$ be the sequence of primes and $\vartheta(x) = \sum_{p \leq x} \log p$, where $p$ runs over the primes not exceeding $x$, be the Chebyshev $\vartheta$-function. In this note we derive lower and upper bounds for $\vartheta(p_n)/n$ by comparing it with $\log p_{n+1}$ and deduce that $\vartheta(p_n)/n=\log p_{n+1}\left(1-\frac{1}{\log n}+\frac{\log\log n}{\log^2 n}\left(1+o(1)\right)\right).$ 
\end{abstract}

\maketitle

\section{Introduction}
\linespread{1.4}

Let $\{p_n\}_{n\ge 1}$ be the sequence of the prime numbers and $\vartheta(x) = \sum_{p \leq x} \log p$, where $p$ runs over the primes not exceeding $x$, be the Chebyshev $\vartheta$-function. The type of bounds that we shall discuss here was introduced by Bonse \cite{bonse}, who showed that $\vartheta(p_n)>2\log p_{n+1}$ holds for every $n\ge 4$ and $\vartheta(p_n)>3\log p_{n+1}$ holds for every $n\ge 5$. Thereafter, P\'osa \cite{posa} showed that, given any $k>1,$ there exists $n_k$ such that $\vartheta(p_n)>k\log p_{n+1}$ holds for all $n\ge n_k.$ Panaitopol \cite{panaitopol} showed that in P\'osa's result we can have $n_k=2k$ and also gave the bound \begin{equation*}\frac{\vartheta(p_n)}{\log p_{n+1}}>n-\pi(n) \quad (n\ge 2),\end{equation*} 
where $\pi(n)$ is equal to the number of primes less or equal to $n$. Hassani \cite{hassani} refined Panaitopol's inequality to the following \begin{equation}\label{hassani}\frac{\vartheta(p_n)}{\log p_{n+1}}>n-\pi(n)\Big(1-\frac{1}{\log n} \Big)\quad (n\geq 101).\end{equation}
Recently, Axler [1, Propositions 4.1 and 4.5] showed that \begin{equation*}\label{axler}1 + \frac{1}{\log p_n} + \frac{2.7}{\log^2p_n} < \log p_n - \frac{\vartheta(p_n)}{n} < 1 + \frac{1}{\log p_n} + \frac{3.84}{\log^2p_n},\end{equation*}
where the left-hand side inequality is valid for every integer $n \geq 218$ and the right-hand side inequality holds for every $n \geq 74004585$. This provides the following asymptotic formula $$\displaystyle\frac{\vartheta(p_n)}{n} = \log p_n - 1 - \frac{1}{\log p_n} + \Theta \Big(\frac{1}{\log^2p_n}\Big).$$ For further terms, see Axler [1, Proposition 2.1].

In the present note, we show the following result, which is a refinement of \eqref{hassani}. 
\begin{theorem} For all $n\ge 6,$ we have
\begin{equation}\label{th1}
    n\Big(1-\frac{1}{\log n}+\frac{\log\log n}{4 \log^2 n}\Big)\le\frac{\vartheta(p_n)}{\log p_{n+1}}\le n\Big(1-\frac{1}{\log n}+\frac{\log\log n}{\log^2 n}\Big).
\end{equation}The left-hand side inequality also holds for $2\le n\le 6.$
\end{theorem}

We also generalise the left-hand side of \eqref{th1} to have the following result.
\begin{theorem}
For every $0<\varepsilon<1,$ there exists $n_\varepsilon\in \mathbb{N}$ such that for every $n\geq n_{\varepsilon}$ it holds that \begin{equation}\label{th1.2}
    n\Big(1-\frac{1}{\log n}+(1-\varepsilon)\frac{\log\log n}{ \log^2 n}\Big)\le\frac{\vartheta(p_n)}{\log p_{n+1}}\le n\Big(1-\frac{1}{\log n}+\frac{\log\log n}{\log^2 n}\Big).
    \end{equation}
\end{theorem}
\begin{corollary}
We have $\displaystyle\frac{\vartheta(p_n)}{n}=\log p_{n+1}\left(1-\frac{1}{\log n}+\frac{\log\log n}{\log^2 n}\left(1+o(1)\right)\right).$
\end{corollary}

\section{Preliminaries}

Define $ G(n,a)=\log n+\log\log n-1+\frac{\log\log n-a}{\log n}$. We shall use the following bounds for $\vartheta(p_n)/n$. 

 \begin{lemma}
For every $n \geq 3$, we have
\begin{equation}
\frac{\vartheta(p_n)}{n} \geq G(n,2.1454), \label{2.2}
\end{equation}
and for every $n \geq 198$, we have
\begin{equation}
\frac{\vartheta(p_n)}{n} \leq G(n,2). \label{2.3}
\end{equation}
\end{lemma}
\begin{proof}
The inequality \eqref{2.2} is due to Robin \cite{robin}, and the inequality \eqref{2.3} was given by Massias and Robin \cite{massias-robin}.
\end{proof}

\begin{lemma} For every $n\ge 227$, we have
\begin{equation}\label{2.4}
     p_n\le n(\log n+\log\log n-0.8),
 \end{equation}and for every $n\ge 2$,
\begin{equation}\label{Dusart}
p_n\ge n(\log n+\log\log n-1).
\end{equation}\end{lemma}
 \begin{proof}
 For $n\ge 8602$, we have the following stronger bound
 \begin{equation}\label{massias-robin2}
     p_n\le n(\log n+\log\log n-0.9385)
 \end{equation}
 given by Massias and Robin \cite{massias-robin}. For $227\le n\le 8601$ we verify the inequality \eqref{2.4} by direct computation. The inequality (\ref{Dusart}) is due to Dusart \cite{dusart}.
 \end{proof}

For the sake of brevity, we shall define $\displaystyle \mathcal{F}(n,\lambda)=1-\frac{1}{\log n}+\lambda\frac{\log\log n}{\log^2 n}$ and rewrite (\ref{th1}) as \begin{equation}\label{th2}
\mathcal{F}(n,0.25)\log p_{n+1}\le\vartheta(p_n)/n\le\mathcal{F}(n,1)\log p_{n+1}\end{equation}
and rewrite \eqref{th1.2} as \begin{equation}\label{new}
\mathcal{F}\left(n,{1-\varepsilon}\right)\log p_{n+1}\le\vartheta(p_n)/n\le\mathcal{F}(n,1)\log p_{n+1}.\end{equation}

\section{Proof of Theorem 1}

The proof of Theorem 1 is split into two lemmas. In the first lemma, we give lower and upper bounds for $\log p_{n+1}.$ 
\begin{lemma}\label{mybd} For every $n\ge 140$, we have
    \begin{equation}\label{lemma3.1}
        \log p_{n+1}<\log n+\log\log n+\frac{\log\log n-0.8+0.018}{\log n}=U(n),
    \end{equation}and for every $n \geq 2$, we have
    \begin{equation}\label{lemma3.2}
\log p_{n+1}>\log n+\log\log n+\frac{\log\log n-1}{\log n+0.5(\log\log n-1)}=V(n).
    \end{equation}
\end{lemma}
\begin{proof} First, we show that for every $x\geq 1$
\begin{equation}\label{logineq}
    \frac{1}{x+0.4}> \log\left(1+\frac{1}{x}\right)>\frac{1}{x+0.5}.
\end{equation}
In order to prove this, we set $ f_a(x)=\log(1+x)-\frac{x}{1+ax}$ and note that, $ f_a'(x)=\frac{x(a^2x+2a-1)}{(1+x)(1+ax)^2}.$ Hence, $f'_{0.4}(x)<0$ for every $x\in(0,1.25)$ which yields $f_{0.4}(1/x)<f_{0.4}(0)=0$ for every $x\ge 1.$
On the other hand, $f'_{0.5}(x)>0$ for all positive $x$, which gives $f_{0.5}(1/x)>f_{0.5}(0)=0$ for every $x\ge 1.$ This completes the proof of (\ref{logineq}).
 
 Next, we give a proof of \eqref{lemma3.1}.  By \eqref{2.4}, we have for $n \geq 227$, \begin{equation}\label{3.4}   \log p_{n+1} \leq \log n+1 + \log ( \log n+1 + \log \log n+1 - 0.8).\end{equation}
 The left-hand side inequality of (\ref{logineq}) implies $\displaystyle\log(n+1)<\log n+\frac{1}{n+0.4}.$ Using (\ref{logineq}) once again, we get\begin{equation*}
    \log\log (n+1)<\log\log n+\log\Big(1+\frac{1}{(n+0.4)\log n}\Big)<\log\log n+\frac{1}{(n+0.4)\log n}.
\end{equation*} 
 Applying this to (\ref{3.4}), we obtain for $n \geq 227$,
 \begin{equation}\label{3.5}
     \log p_{n+1}<\log n+\log\log n+\frac{\log\log n-0.8}{\log n}+\frac{1}{\log n}\cdot\frac{\log n+1+1/\log n}{n+0.4}.
 \end{equation}
 Now, $g(x)=\displaystyle\frac{\log x+1+1/\log x}{x+0.4}$ is a decreasing function for $x\ge2$ with $g(e^{5.99})\le 0.018$. Hence $g(x) \leq 0.018$ for every $x \geq 400 > e^{5.99}$. Combined with (\ref{3.5}), it shows out that $\log p_{n+1} < U(n)$ for every $n \geq 400$. For every $140\leq n\leq 399$ we check the inequality \eqref{lemma3.1} with a computer. This completes the proof of (\ref{lemma3.1}).
 
 To prove the inequality \eqref{lemma3.2}, first note that \eqref{Dusart} gives for every $n\ge 1$,\begin{equation}\label{3.6}
 \log p_{n+1}\ge\log(n+1)+\log(\log(n+1)+\log\log(n+1)-1).
 \end{equation}The right-side inequality of (\ref{logineq}) gives $\displaystyle\log(n+1)>\log n+\frac{1}{n+0.5}.$ Using (\ref{logineq}) once again, we get, for $n\ge2,$
\begin{align*}
    \log\log(n+1)-\log\log n>\log\Big(1+\frac{1}{(n+0.5)\log n}\Big)>\frac{1}{(n+0.5)\log n+0.5}.
\end{align*}Applying this to \eqref{3.6}, we arrive at
\begin{align*}
 \log p_{n+1}&>\log n+\log\Big(\log n+\frac{1}{n+0.5}+\log\log n+\frac{1}{(n+0.5)\log n+0.5}-1\Big)\\
  &>\log n+\log\log n+\log\Big(1+\frac{\log\log n-1}{\log n}\Big).\end{align*}Applying (\ref{logineq}) one more time, we get $\log p_{n+1}>V(n)$ for every $n\ge 2.$\end{proof}

\vspace{-3mm}\begin{lemma}\label{lm4} For every $n\ge 396$, we have
\begin{equation}\label{lemma4.1}
     G(n,2.1454)\ge \mathcal{F}(n,0.25)\cdot U(n),\end{equation}
 and for every $n\ge 2$, we have    \begin{equation}\label{lemma4.2}
  G(n,2)\le \mathcal{F}(n,1)\cdot V(n).
 \end{equation} Here $U(n)$ and $V(n)$ are defined as in Lemma \ref{mybd}. \end{lemma} 
\begin{proof} We start with the proof of (\ref{lemma4.1}). Setting $x = \log n$, the inequality (\ref{lemma4.1}) can be rewritten as
\begin{equation*}
    x+\log x-1+\frac{\log x-2.1454}{x}\ge \Big(1-\frac{1}{x}+\frac{\log x}{4x^2}\Big)\Big(x+\log x+\frac{\log x-0.8+0.018}{x}\Big),
\end{equation*}which is equivalent to
\begin{equation*}
    \Big(\frac{3}{4}\log x+\frac{\log x}{x}\Big)+\Big(-2.1454-\frac{\log^2 x}{4x}-\frac{\log^2 x}{4x^2}\Big)+(0.8-0.018)\Big(1-\frac{1}{x}+\frac{\log x}{4x^2}\Big)\geq 0.
\end{equation*}
The left-hand side is a sum of three increasing functions on the 
interval $[5.7, \infty)$ and at $x = 5.99$ the left-hand side is positive. 
So the last inequality holds for every $x \geq 5.99$; i.e., for every $n 
\geq 400$. A direct computation shows that the inequality (\ref{lemma4.1}) also 
holds for every $n$ satisfying $396 \leq n \leq 399$.

Next, we give a proof of (\ref{lemma4.2}). It is easy to see that
  $$x^2 + \log x(\log x - 1) > \frac{x}{2}\log x(\log x - 1)$$
for every $x> 0$. Now, for $x\ge 1$, the last inequality is seen to be equivalent to
 $$ \left( 1 - \frac{1}{x} + \frac{\log x}{x^2} \right) \frac{\log x - 1}{x + 0.5(\log x-1)} \geq \frac{\log x-2}{x}.$$
Since $\frac{\log^2 x}{x^2} \geq 0$ for every $x > 0$, we get
\begin{equation}\label{3.9}  \frac{\log^2 x}{x^2} + \left( 1 - \frac{1}{x} + \frac{\log x}{x^2} \right) \frac{\log x - 1}{x + 0.5(\log x-1)} \geq \frac{\log x-2}{x}\end{equation}
for every $x \ge 1$.  Substituting $x = \log n$ in (\ref{3.9}), we obtain the inequality (\ref{lemma4.2}) for every integer $n \geq 3$. We can directly check that (\ref{lemma4.2}) holds for $n=2$ as well.\end{proof}

Finally, we give a proof of Theorem 1.
\begin{proof}[Proof of Theorem 1.] We use (\ref{2.2}), (\ref{lemma4.1}) and (\ref{lemma3.1}) respectively to see that for every $n\ge 396,$
 \begin{equation*}
    \vartheta(p_n)/n\ge G(n,2.1454)\ge \mathcal{F}(n,0.25)U(n)>\mathcal{F}(n,0.25)\log p_{n+1}. \end{equation*}
A direct computation shows that the left-hand side inequality of (\ref{th2}) also holds for every integer $n$ with $2 \leq n \leq 395$.
 
 In order to prove the right-hand side inequality of (\ref{th2}), we combine (\ref{2.3}), (\ref{lemma4.2}) and (\ref{lemma3.2}) respectively to get  \begin{equation*}\vartheta(p_n)/n\le G(n,2)\le \mathcal{F}(n,1)V(n)\le\mathcal{F}(n,1)\log p_{n+1}
 \end{equation*}for every $n\ge 198$. For smaller values of $n$, we use a computer. 
 \end{proof}

\section{Proof of Theorem 2}

The right-hand side of \eqref{new} has been established already. To show the left-hand side, we start with the following lemma.
\begin{lemma}\label{lm5} For any $0<\varepsilon<1$, there exists $m_\varepsilon\in\mathbb{N}$ such that
\begin{equation}\label{lemma last}
     G(n,2.1454)\ge \mathcal{F}(n,1-\varepsilon)\cdot U(n)\end{equation}holds for every $n\geq m_\varepsilon.$ Here $U(n)$ is defined as in Lemma \ref{mybd}. \end{lemma} 
     \begin{proof} Fix any $0<\varepsilon<1$. We denote $a=2.1454, b=0.8-0.018$ and set $x = \log n$ to transform the inequality (\ref{lemma last}) into 
\begin{equation*}
    x+\log x-1+\frac{\log x-a}{x}\ge \Big(1-\frac{1}{x}+(1-\varepsilon)\frac{\log x}{4x^2}\Big)\Big(x+\log x+\frac{\log x-b}{x}\Big).
\end{equation*}This is equivalent to
\begin{equation*}
    \left(\varepsilon\log x+\frac{\log x}{x}\right)+\left(-a-(1-\varepsilon)\frac{\log^2 x}{x^2}(x+1)\right)+b\left(1-\frac{1}{x}+(1-\varepsilon)\frac{\log x}{x^2}\right)\geq 0.
\end{equation*}
Now, the left-hand side is a sum of three functions, each of which is strictly increasing for all sufficiently large $x,$ and the limit of the left-hand side, as $x\to\infty$, is $+\infty.$ Therefore we conclude that the last inequality holds for all sufficiently large $x.$ \end{proof}

\begin{proof}[Proof of Theorem 2.] For any $0<\varepsilon<1,$ we have $m_\varepsilon\in\mathbb{N}$ such that (\ref{lemma last}) holds for every $n\ge m_\varepsilon$. We combine this with (\ref{2.2}) and (\ref{lemma3.1}) to obtain that for every $n\geq n_\varepsilon := \max\{m_\varepsilon,140\}$
 \begin{equation*}
    \vartheta(p_n)/n\ge G(n,2.1454)\ge \mathcal{F}(n,1-\varepsilon)U(n)\geq\mathcal{F}(n,1-\varepsilon)\log p_{n+1}. \end{equation*}This completes the proof.\end{proof}

 \section{Remarks}
 \begin{enumerate}
 \item For every $n\ge 599$, we have \[\frac{\pi(n)}{n}\ge \frac{1}{\log n}+\frac{1}{\log^2 n},\]
which was found by Dusart \cite{dusart2}. Using this and a computer, we get
\begin{displaymath}
\frac{\pi(n)}{n} \geq \frac{1}{\log n - 1} \left( 1 - \frac{\log \log n}{4 \log n} \right)
\end{displaymath}
for every integer $n \geq 83$. Hence, \eqref{th1} is an improvement of \eqref{hassani}.

\vspace{2mm}
\item The bounds given in \eqref{th1} are particularly useful for comparing $\vartheta(p_n)/n$ with $\log p_{n+1}.$ To see a numerical example, we use a computer to find that for $n\ge 23$ the relative error in approximating $\vartheta(p_n)/n$ with $\mathcal{F}(n,0.25)$ is less than $5\%$ and for $n\ge 114$ it is less than $2\%.$ An important feature of \eqref{th1} is that it holds even for very small values of $n.$ 

 \end{enumerate}

\section*{Acknowledgements} 
I am thankful to Mridul Nandi
(Indian Statistical Institute, Kolkata, India) and Mehdi Hassani (University of Zanjan, Iran) for their valuable suggestions.

\makeatother

\end{document}